\documentclass[a4paper,12pt]{article}

\usepackage{latexsym}
\usepackage{amssymb}
\usepackage{amsmath}
\usepackage{amsthm}
\usepackage{enumerate}
\usepackage{bm}
\usepackage{color}
\usepackage{comment}

\def\A{\mathcal{A}}
\def\B{\mathcal{B}}
\def\R{\mathbb{R}}
\def\C{\mathbb{C}}
\def\K{\mathbb{K}}
\def\Z{\mathbb{Z}}
\def\c{\mathbf{c}}

\def\p#1{\frac{\partial}{\partial #1}}

\def\Co#1{\mathrm{Cox}(#1)}
\def\Is#1{\mathrm{Ish}(#1)}
\def\Sh#1{\mathrm{Shi}(#1)}
\def\I{\mathcal{I}}
\def\Ch{\mathrm{Ch}}

\DeclareMathOperator{\rank}{rank}

\DeclareMathOperator{\Der}{Der}

\newtheorem{theorem}{Theorem}[section]

\newtheorem{cor}[theorem]{Corollary}
\newtheorem{lemma}[theorem]{Lemma}
\newtheorem{define}[theorem]{Definition}

\newcommand{\BigFig}[1]{\parbox{12pt}{\Huge #1}}
\newcommand{\BigZero}{\BigFig{0}}

\title{The freeness of Ish arrangements}

\author{
Takuro Abe\thanks{Department of Mechanical Engineering and Science, Kyoto University, Kyoto 606-8501, Japan. 
e-mail:abe.takuro.4c@kyoto-u.ac.jp}, 
Daisuke Suyama\thanks{Department of Mathematics, Hokkaido University, Sapporo, Hokkaido 060-0810, Japan.
email:dsuyama@math.sci.hokudai.ac.jp}, 
Shuhei Tsujie\thanks{Department of Mathematics, Hokkaido University, Sapporo, Hokkaido 060-0810, Japan.
email:tsujie@math.sci.hokudai.ac.jp}
}

\date{}

\begin{document}

\maketitle

\begin{abstract}
The Ish arrangement was introduced by Armstrong to give
a new interpretation of the $q,t$-Catalan numbers of
Garsia and Haiman.
Armstrong and Rhoades showed that there are some
striking similarities between the Shi arrangement and
the Ish arrangement and posed some problems.
One of them is whether the Ish arrangement is a free arrangement
or not.
In this paper, 
we verify that the Ish arrangement is supersolvable
and hence free.
Moreover, we give a necessary and sufficient condition
for the deleted Ish arrangement to be free. 
\end{abstract}



\section{Introduction} 

Let $ \K $ be a field of characteristic $ 0 $ and
$\{ x_{1},\ldots ,x_{\ell} \}$ a basis for the dual space 
$ (\K^{\ell})^{*} $ of the $ \ell $-dimensional vector space $ \K^{\ell} $.
The \textbf{Coxeter arrangement} $\Co{\ell}$ 
of type $A_{\ell-1}$
(also called the \textbf{braid arrangement}) is
\[
\Co{\ell} 
:= \left\{ \{ x_{i}-x_{j}=0 \} \mid 1 \leq i<j \leq \ell \right\},
\]
where $\{ x = k \}\ (x \in (\K^{\ell})^{*},\ k \in \K)$ 
is the affine hyperplane $\{ v \in \K^{\ell} \mid x(v)=k \}$.
Then
the \textbf{Shi arrangement} $\Sh{\ell}$ and
the \textbf{Ish arrangement} $\Is{\ell}$
are defined by
\begin{align*}
\Sh{\ell} &:=
\Co{\ell} \cup 
\left\{ \{ x_{i}-x_{j} = 1 \} \mid 1 \leq i<j \leq \ell \right\}, \\
\Is{\ell} &:=
\Co{\ell} \cup 
\left\{ \{ x_{1}-x_{j} = i \} \mid 1 \leq i<j \leq \ell \right\}.
\end{align*}
The Shi arrangement originally defined over $\R$ was introduced by
J.Y. Shi \cite{skazhdan} in the study of the Kazhdan-Lusztig
representation theory of the affine Weyl groups.
The Ish arrangement also originally defined over $ \R $ was 
introduced by Armstrong in \cite{adiag}.
He gave a new interpretation of the $q,t$-Catalan numbers of
Garsia and Haiman by using these two arrangements.
Armstrong and Rhoades showed that there are
some striking similarities between the Shi arrangement and 
the Ish arrangement in \cite{adiag,arish}. 

Let $\A$ be an arrangement in $\K^{\ell}$. 
Let $L(\A)$ be the set of nonempty intersections of hyperplanes in $\A$,
which is partially ordered by the reverse inclusion of subspaces.  
Define the M\"obius function $\mu : L(\A) \rightarrow \Z$ as follows:
\[
\mu (\K^{\ell}) = 1,
\]
\[ 
\mu (X) = -\sum_{\K^{\ell} \leq Y < X} \mu (Y)
\ \ (X \neq \K^{\ell}).
\]
Then the \textbf{characteristic polynomial} $\chi (\A ,t) \in \Z[t]$
of $\A$ is defined by
\[
\chi (\A ,t) = \sum_{X \in L(\A)} \mu (X) t^{\dim X}.
\]
The following theorem is one of the similarities pointed out by Armstrong.

\begin{theorem}[\cite{adiag,hfamily}]
\label{char}
The characteristic polynomial of the Shi arrangement 
and the Ish arrangement are given by
\[
\chi (\Sh{\ell} ,t)
=\chi (\Is{\ell}, t)
=t(t-\ell )^{\ell -1}.
\] 
\end{theorem}

Let $\{x_{1}, \dots, x_{\ell}, z\}$ be a basis for $V^{*}$ of
$ V := \K^{\ell + 1} $.
Then, as in \cite[Definition 1.15]{otarrangements},
we have the \textbf{cone} $\c(\Is{\ell})$ 
over the Ish arrangement which is a central arrangement
(Namely, an arrangement whose hyperplanes pass through the origin)
in $ V $ defined by
\[
Q\left(\c(\Is{\ell})\right)
=z \prod_{1\leq i<j \leq \ell} (x_{i}-x_{j}) (x_{1}-x_{j}-iz)=0.
\]

Let $S$ be the symmetric algebra of the dual space $V^{*}$.
$S$ can be identified with the polynomial ring 
$\K[x_{1},\ldots ,x_{\ell},z]$. 
Let $\Der(S)$ be the module of derivations of $S$
\begin{multline*}
\Der(S):=
\{ \theta : S \rightarrow S \mid
\theta \text{ is } \K\text{-linear}, \\
\theta(fg)=f\theta(g) + \theta(f)g \text{ for any } f,g \in S \}.
\end{multline*}
Then, for a central arrangement $\A$ in $V$, the module of logarithmic
derivations $D(\A)$ of $\A$ is defined to be
\begin{align*}
D(\A) 
&:=\{ \theta \in \Der(S) \mid \theta (Q(\A)) \in Q(\A) S \} \\
&=\{ \theta \in \Der(S) \mid 
\theta (\alpha_{H}) \in \alpha_{H} S \text{ for any } H \in \A \},
\end{align*}
where $Q(\A)$ is the defining polynomial of $\A$
and $\alpha_{H}$ is a linear form such that $\ker (\alpha_{H}) = H$.
We say that $\A$ is \textbf{free} if $D(\A)$ is a free $S$-module.
Then $D(\A)$ has a homogeneous basis 
$\{ \theta_{0}, \ldots ,\theta_{\ell} \}$
and the tuple of degrees 
$\exp \A = (\deg \theta_{0}, \ldots ,\deg \theta_{\ell})$
is called the \textbf{exponents} of $\A$. 

The main purpose of this paper is to settle a problem of whether the Ish arrangements are free or not, 
which was posed by Armstrong and Rhoades in 
\cite[p.\hspace{0.3ex}1527, (3)]{arish}.
We define a new class of arrangements which is a generalization of
the Ish arrangements and will characterize free arrangements in this class. 

\begin{define}
Let $N=(N_{2},N_{3},\ldots ,N_{\ell})$ be a tuple of
finite subsets $ N_{j} $ in $ \K $.
Define the \textbf{$N$-Ish arrangement} $\Is{N}$ by
\begin{multline*}
\Is{N} :=
\left\{ \{ x_{1}-x_{j} = a  \} \mid 2 \leq j \leq \ell, a \in N_{j} \right\}\\
\cup \left\{ \{x_{i} - x_{j} = 0 \} \mid 2 \leq i < j \leq \ell \right\}.
\end{multline*}
We say that $ N $ is a \textbf{nest} if there exists a permutation $ w $ of $ \{2, \dots, \ell\} $ such that
\begin{align*}
N_{w(2)} \subseteq N_{w(3)} \subseteq \dots \subseteq N_{w(\ell)}. 
\end{align*}
\end{define}

In particular, when $ N_{j} = \{0, 1, \dots, j - 1 \} $ for each $ j $,
the $N$-Ish arrangement $ \Is{N} $ is the Ish arrangement $\Is{\ell}$.
We denote the cone over the $N$-Ish arrangement 
$\c(\Is{N})$ by $\I=\I_{N}$.
The defining polynomial of $\I$ can be expressed as
\[
Q(\I)=z
\left(
\prod_{j=2}^{\ell} \prod_{a \in N_{j}}
(x_{1}-x_{j}-az)
\right)
\left(
\prod_{2 \leq i < j \leq \ell}
(x_{i}-x_{j})
\right).
\]

Our main results are as follows:
\begin{theorem}
\label{maintheoA}
The following four conditions are equivalent: 
\begin{enumerate}
\item $ N $ is a nest. 
\item $ \I_{N} $ is supersolvable. 
\item $ \I_{N} $ is inductively free.
\item $ \I_{N} $ is free. 
\end{enumerate}
\end{theorem}

The definitions of supersolvable and inductively free arrangements will be mentioned in Section 2. 
Note that the implications $ (2) \Rightarrow (3) \Rightarrow (4) $ are general properties for arrangements \cite{otarrangements}. 
This theorem asserts that there are no differences among these properties for $ N $-Ish arrangements. 

\begin{theorem}
\label{maintheoB}
Let $ N = (N_{2}, N_{3}, \dots, N_{\ell}) $ with $ N_{2} \subseteq N_{3} \subseteq \dots \subseteq N_{\ell} $. 
Define homogeneous derivations $ \theta_{0}, \theta_{1}, \dots, \theta_{\ell} $ by
\begin{align*}
\theta_{0} &:= \sum_{i=1}^{\ell} \p{x_{i}}, \qquad
\theta_{1} = \sum_{i=1}^{\ell}x_{i}\frac{\partial}{\partial x_{i}} + z\frac{\partial}{\partial z},   \\
\theta_{k} &:= \sum_{s=2}^{k}
\left(
\prod_{a \in N_{k}} (x_{1}-x_{s}-az)\prod_{t=k+1}^{\ell}(x_{s}-x_{t})
\right)
\frac{\partial}{\partial x_{s}}
\ \ (2 \leq k \leq \ell).
\end{align*}
Then $\theta_{0},\theta_{1}, \ldots ,\theta_{\ell}$ form
a basis for $D(\I_{N})$.
In particular, the exponents are given by
\begin{align*}
\exp \I_{N}=
(0,\ 1,\ |N_{2}|+ \ell -2,\ |N_{3}|+ \ell -3,\ \ldots \ ,\ |N_{\ell}|), 
\end{align*}
where $ |N_{j}| $ denotes the cardinality of $ N_{j} $. 
\end{theorem}

\begin{cor}
\label{maincor}
The cone over the Ish arrangement $\c(\Is{\ell})$
is free with exponents
$\exp (\c(\Is{\ell}))=
(0,\ 1,\ 
\underbrace{\ell,\ \ell,\ \ldots \ ,\ \ell}_{(\ell-1) \text{ times}})$.
Moreover the homogeneous derivations
\begin{align*}
\theta_{0} &= \sum_{i=1}^{\ell} \p{x_{i}}, \qquad
\theta_{1} = \sum_{i=1}^{\ell}x_{i}\frac{\partial}{\partial x_{i}} + z\frac{\partial}{\partial z},   \\
\theta_{k} &= \sum_{s=2}^{k}
\left(
\prod_{i=0}^{k-1} (x_{1}-x_{s}-iz)\prod_{t=k+1}^{\ell}(x_{s}-x_{t})
\right)
\frac{\partial}{\partial x_{s}}
\ \ (2 \leq k \leq \ell) 
\end{align*}
form a basis for $ D(\c(\Is{\ell})) $. 
\end{cor}

If an arrangement $\A$ is a free arrangement,
then the characteristic polynomial of $\A$ 
can be expressed by using its exponents:

\begin{theorem}[\cite{tgeneralized}]
\label{factorization}
If an arrangement $\A$ is free with exponents
$(d_{1},\ldots ,d_{\ell})$, then the characteristic polynomial
of $\A$ splits as
\[
\chi (\A ,t) = \prod_{i=1}^{\ell} (t-d_{i}).
\]
\end{theorem}
Since we have the relation between
the characteristic polynomials of $\A$ and $\c\A$
\[
\chi (\c\A ,t) =(t-1)\chi (\A ,t),
\]
we obtain a new proof of Theorem \ref{char} 
from Corollary \ref{maincor} and Theorem \ref{factorization}.

The complement $ M(\A) := \K^{\ell} \setminus \cup_{H \in \A}H $ of a supersolvable arrangement $ \A $ has very interesting properties:
If $ \K = \C $, the complement $ M(\A) $ is fiber type \cite{tmodular}. 
In particular, $ M(\A) $ is a $ K(\pi, 1) $ space, i.e., the homotopy groups $ \pi_{i}(M(\A)) = 0 $ for $ i \geq 2 $. 
When $ \K = \R $, the complement $ M(\A) $ is a disjoint union of chambers. 
For chambers $ C, C^{\prime} $, define $ d(C, C^{\prime}) $ by the number of hyperplanes in $ \A $ separating $ C $ from $ C^{\prime} $. 
Bj\"{o}rner, Edelman, and Ziegler \cite{bezhyperplane} gave the wall-crossing formula as follows: 
There exists a base chamber $ B $ of $ \A $ such that
\begin{align*}
\sum_{C \in \Ch(\A)}t^{d(B,C)} = \prod_{i=1}^{\ell}(1+t+\dots + t^{d_{i}}), 
\end{align*}
where $ (d_{1}, \dots, d_{\ell}) $ is the exponents of $ \A $ and $ \Ch(\A) $ denotes the set of all chambers of $ \A $. 
Therefore, we derive the following corollary from our main theorems 
\ref{maintheoA} and \ref{maintheoB}.

\begin{cor}
\label{corss}
Let $ N = (N_{2}, N_{3}, \dots, N_{\ell}) $ be a nest. 
\begin{enumerate}
\item
If $\K=\C$, then the complement $M(\I_{N})$ 
of the cone over the $N$-Ish arrangement $\I_{N}$ is $K(\pi ,1)$.
\item
If $\K=\R$, then there exists a base chamber $ B \in \Ch (\I_{N}) $ 
such that
\[
\sum_{C \in \Ch(\I_{N})}t^{d(B,C)} 
=(1+t) \prod_{i=2}^{\ell}(1+t+\dots + t^{|N_{i}|+\ell -i}).
\]
\end{enumerate}
\end{cor}

The organization of this paper is as follows.
In Section 2, we review the theory of supersolvable arrangements and prove Theorem \ref{maintheoA}. 
In Section 3, we give an explicit expression of a base chamber as mentioned in Corollary \ref{corss} (2) for $ \I_{N} $ with $ N $ nested. 
In Section 4, we verify Theorem \ref{maintheoB} applying Saito's criterion. 
In Section 5, we recall the deleted arrangement $ \Sh{G} $ and $ \Is{G} $ defined by Armstrong and Rhoades in \cite{arish} 
and prove that $ \Sh{G} $ and $ \Is{G} $ share the freeness.

\section{Supersolvability and freeness of $ \I $}
For an arrangement $ \A $, let $ L(\A) $ be the set of nonempty intersections of hyperplanes in $ \A $. 
If $ \A $ is central, then $ L(\A) $ is a geometric lattice with the order by reverse inclusion: 
$ X \leq Y \Leftrightarrow Y \subseteq X $. 
In the rest of this section, ``arrangement" means ``central arrangement". 
The rank of an arrangement $ \A $, denoted by $ \rank(\A) $, is the codimension of $ \cap_{H \in \A}H $. 
We say that $ \A $ is essential if $ \rank(\A) $ is equal to the dimension of the ambient space of $ \A $. 

An arrangement $ \A $ is supersolvable if the intersection lattice $ L(\A) $ is supersolvable as defined by Stanley \cite{ssuper}.
The following lemma is widely known. 

\begin{lemma}[{\cite[Theorem 4.3]{bezhyperplane}}]
\label{lemmaSS}
An arrangement $ \A $ is supersolvable if and only if
there exists a filtration
\begin{align*}
\A = \A_{\ell} \supseteq \A_{\ell-1} \supseteq \dots \supseteq \A_{1}
\end{align*}
such that
\begin{enumerate}
\item $ \rank(\A_{i}) = i \; (i=1,2,\dots, \ell) $.
\item For any $ H, H^{\prime} \in \A_{i} $ with $ H \neq H^{\prime} $, there exists some $ H^{\prime\prime} \in \A_{i-1} $ such that $ H \cap H^{\prime} \subseteq H^{\prime\prime} $. 
\end{enumerate}
\end{lemma}

Let $ \A $ be an arrangement. 
For a hyperplane $ H \in \A $, define arrangements
\begin{align*}
\A^{\prime} := \A \setminus \{H\} \text{ and } 
\A^{\prime\prime} := \left\{H^{\prime} \cap H \mid H^{\prime} \in \A^{\prime} \right\}. 
\end{align*}
The tuple $ (\A, \A^{\prime}, \A^{\prime\prime}) $ is called the triple of arrangements with respect to $ H $.
For a triple $ (\A, \A^{\prime}, \A^{\prime\prime}) $, the Addition Theorem \cite{tarrangementsI,tarrangementsII} asserts that if $ \A^{\prime} $ and $ \A^{\prime\prime} $ are free and $ \exp \A^{\prime\prime} \subset \exp \A^{\prime}$, then $ \A $ is free. 
\begin{define}
Define the inductive freeness by the following: 
\begin{enumerate}
\item The empty arrangement is inductively free. 
\item $ \A $ is inductively free if there exists $ H \in \A $ such that $ \A^{\prime} $ and $ \A^{\prime\prime} $ are inductively free and $ \exp \A^{\prime\prime} \subset \exp \A^{\prime} $. 
\end{enumerate}
\end{define}
Thanks to the Addition Theorem, the inductive freeness implies the freeness. 
Moreover, it is also known that the the supersolvability implies the inductive freeness
(see \cite[Theorem 4.58]{otarrangements} for example). 

We will use the following lemma which is a part of the Addition-Deletion Theorem:
\begin{lemma}[{\cite[Theorem 4.46]{otarrangements}}]
\label{lemmaVAD3}
Let $ (\A, \A^{\prime}, \A^{\prime\prime}) $ be a triple. 
Suppose that $ \A $ is an essential arrangement of rank $ 3 $
and that arrangements $ \A^{\prime} $ and $ \A^{\prime\prime} $ are free with $ \exp(\A^{\prime}) = (1,a,b) $ and $ \exp(\A^{\prime\prime}) = (1,c) $. 
If $ c \not\in \{a,b\} $ then $ \A $ is not free. 
\end{lemma}

We are now prepared to prove Theorem \ref{maintheoA}. 

\begin{proof}[Proof of Theorem \ref{maintheoA}]
$ (1) \Rightarrow (2) $
Without loss of generality, we may assume that $ N_{2} \supseteq N_{3} \supseteq \dots \supseteq N_{\ell} $. 
If $N_{2} = \emptyset$, i.e., $N_{2}=N_{3}=\cdots =N_{\ell}=\emptyset$,
then $\I$ is the Coxeter arrangement $\Co{\ell-1}$ 
of the type $A_{\ell-2}$ which is supersolvable
(see \cite[Example 2.33]{otarrangements} or \cite[Example 2.6]{ssuper}).
Assume that $N_{2} \neq \emptyset$.
For each $ i \in \{1,2, \dots, \ell\} $, define $ X_{i} \in L(\I) $ by
\[
X_{i} := \{ z=0 \} \cap \bigcap_{j=2}^{i} \{ x_{1}-x_{j}=0 \}
\cap \bigcap_{2 \leq j <k \leq i} \{ x_{j}-x_{k}=0 \}. 
\]
Then the rank of the localization $ \I_{i} := \I_{X_{i}} = \left\{ H \in \I \mid H \supseteq X_{i} \right\} $ is equal to $ i $ and we have 
\begin{multline*}
\I_{i} 
= \left\{ \{x_{1}-x_{j} = az \} \mid 2 \leq j \leq i \text{ and } a \in N_{j} \right\} \\
\cup \left\{ \{ x_{j} - x_{k} = 0 \} \mid 2 \leq j < k \leq i \right\}
\cup \left\{ \{ z = 0\} \right\}
\end{multline*}
Hence there exists a filtration
\begin{align*}
\I = \I_{\ell} \supseteq \I_{\ell-1} \supseteq \dots \supseteq \I_{1}.
\end{align*}
By Lemma \ref{lemmaSS}, we have only to verify that for any $ H,H^{\prime} \in \I_{i} $ with $ H \neq H^{\prime} $ there exists some $ H^{\prime\prime} \in \I_{i-1} $ such that $ H \cap H^{\prime} \subseteq H^{\prime\prime} $ for each $ i \in \{2, \dots, \ell\} $. 
We may assume that both $ H $ and $ H^{\prime} $ do not belong to $ \I_{i-1} $. 
Then $ H $ and $ H^{\prime} $ belong to
\begin{align*}
\I_{i} \setminus \I_{i-1} = \left\{ \{x_{1}-x_{i} = az\} \mid a \in N_{i} \right\} \cup \left\{ \{ x_{j} -x_{i} = 0 \}\mid 2 \leq j < i \right\}. 
\end{align*}
First, let $ a $ and $ b $ be distinct elements in $ N_{i} $.
Suppose that $ H=\{x_{1}-x_{i} = az\} $ and $ H^{\prime}=\{x_{1}-x_{i}=bz\} $. 
Then $ H \cap H^{\prime} \subseteq \{z=0\} \in \I_{i-1} $. 
Second, let $ j $ and $ k $ be distinct integers in $ \{2, \dots, i-1 \} $. 
Assume that $ H = \{ x_{j} - x_{i} = 0 \} $ and $ H^{\prime} = \{x_{k} - x_{i} = 0\} $. 
Then $ H \cap H^{\prime} \subseteq \{x_{j} - x_{k} = 0 \} \in \I_{i-1} $. 
Finally, let $ H=\{x_{1}-x_{i}=az\} $ and $ H^{\prime} = \{x_{j}-x_{i} = 0\} $ with $ a \in N_{i} $ and $ 2 \leq j < i $. 
Then $ H \cap H^{\prime} \subseteq \{x_{1} - x_{j} = az\} \in \I_{i-1}$ by the assumption $ a \in N_{i} \subseteq N_{j} $. 
Thus the cone over the $ N $-Ish arrangement $ \I $ is supersolvable. 

$ (2) \Rightarrow (3) \Rightarrow (4) $ We have nothing to prove as mentioned before. 

$ (4) \Rightarrow (1) $
When $\ell = 2$, the tuple $N=( N_{2} )$ is a nest.
For $\ell \geq 3$, we will prove that
if $ N $ is not a nest then $\I$ is not free
by induction on $\ell$.
First, let $ \ell = 3 $. 
Then we have $ N = (N_{2},N_{3}) $. 
Let $ H \in \I $ be the hyperplane $ \{ x_{2}-x_{3} = 0 \} $ and $ (\I, \I^{\prime}, \I^{\prime\prime}) $ the triple with respect to $ H $. 
One can verify easily that the homogeneous derivations
\begin{align*}
\sum_{i=1}^{3}x_{i}\frac{\partial}{\partial x_{i}} + z\frac{\partial}{\partial z}, \quad
\prod_{a \in N_{2}}(x_{1}-x_{2}-az)\frac{\partial}{\partial x_{2}}, \quad
 \prod_{a \in N_{3}}(x_{1}-x_{3}-az)\frac{\partial}{\partial x_{3}}
\end{align*}
form a basis for $ D(\I^{\prime}) $ (with the non-essential derivation $ \sum_{i=1}^{3}\frac{\partial}{\partial x_{i}} + \frac{\partial}{\partial z} $). 
Hence the arrangement $ \I^{\prime} $ is free with exponents $ (1,|N_{2}|,|N_{3}|) $. 
The arrangement $ \I^{\prime\prime} $ is also free with exponents $ (1,|N_{2} \cup N_{3}|) $ since $ \rank(\I^{\prime\prime}) = 2 $ and $ |\I^{\prime\prime}| = 1 + |N_{2}\cup N_{3}| $. 
By the assumption, $ N $ is not a nest, i.e., $ N_{2} \not\subseteq N_{3} $ and $ N_{2} \not\supseteq N_{3} $, hence we have that $ |N_{2} \cup N_{3}| $ is strictly larger than both of $ |N_{2}| $ and $ |N_{3}| $. 
Therefore, by Lemma \ref{lemmaVAD3}, we have concluded that $ \I $ is not free. 

Now suppose that $ \ell > 3 $. 
Since $ N $ is not a nest, there exist integers $ i,j $ such that $ N_{i} \not\subseteq N_{j} $ and $ N_{i} \not\supseteq N_{j} $. 
Define $ X \in L(\I) $ by
\begin{align*}
X := \left\{ z=x_{1}-x_{i}=x_{1}-x_{j}=0 \right\}. 
\end{align*}
Then we have
\begin{multline*}
\I_{X} = 
\left\{ \{ x_{1}-x_{k} = az \} \mid k \in \{i,j\} \text{ and } a \in N_{k} \right\} 
\cup \left\{ \{ x_{i} - x_{j} = 0 \}, \{ z=0 \} \right\}. 
\end{multline*}
Hence $ \I_{X} $ is equivalent to $ \c(\Is{N_{i},N_{j}}) $ discussed 
in the above paragraph. 
Therefore the localization $ \I_{X} $ is not free, neither is $ \I $. 
\end{proof}

\section{A base chamber for $ \I $}
In this section, we set $ \K = \R $ and will give an explicit expression of a base chamber as mentioned in Corollary \ref{corss} (2) for $ \I $. 
Let $ \A $ be an arrangement and $ \Ch(\A) $ the set of all chambers of $ \A $. 
Fixing a chamber $ B \in \Ch(\A) $, we can define a partial order on $ \Ch(\A) $:  $ C \leq C^{\prime} \Leftrightarrow S_{B}(C) \subseteq S_{B}(C^{\prime}) $, where $ S_{B}(C) $ denotes the set of hyperplanes in $ \A $ which separate $ C $ from $ B $. 
Let $ P(\A, B) $ denote the set $ \Ch(\A) $ with this partial order. 

For a subarrangement $ \B \subseteq \A $, define a map $ \pi \colon \Ch(\A) \rightarrow \Ch(\B) $ by $ \pi(C) = C^{\prime} $, 
where $ C^{\prime} \in \Ch(\B) $ is a unique chamber including $ C \in \Ch(\A) $. 
Let $ \A $ be a supersolvable arrangement. 
Then there exists a filtration
\begin{align*}
\A = \A_{\ell} \supseteq \A_{\ell - 1} \supseteq \dots \supseteq \A_{1},
\end{align*}
which satisfies the conditions in Lemma \ref{lemmaSS}. 
This sequence of subarrangements induces a chain of maps: 
\begin{align*}
\Ch(\A_{\ell}) 
\xrightarrow{\pi_{\ell}} \Ch(\A_{\ell - 1}) 
\xrightarrow{\pi_{\ell -1}} \dots 
\xrightarrow{\pi_{2}} \Ch(\A_{1}). 
\end{align*}

\begin{define}
For each $ 1 \leq i \leq \ell $, define that a chamber $ B \in \Ch(\A_{i}) $ is \textbf{canonical} by the following: 
\begin{enumerate}
\item Any chambers in $ \Ch(\A_{1}) $ or $ \Ch(\A_{2}) $ are canonical. 
\item When $ i \geq 3 $, a chamber $ B \in \Ch(\A_{i}) $ is canonical if $ \pi_{i}(B) \in \Ch(\A_{i-1}) $ is canonical and $ \mathcal{F}(B) := \pi_{i}^{-1}(\pi_{i}(B)) $ is linearly ordered in $ P(\A_{i}, B) $. 
\end{enumerate}
\end{define}

Bj\"{o}rner, Edelman, and Ziegler gave the wall-crossing formula with a canonical chamber. 
\begin{theorem}[{\cite[Theorem 4.4]{bezhyperplane}} ]
Let $ \A $ be a supersolvable arrangement and $ B \in \Ch(\A) $ a canonical chamber. 
Then
\begin{align*}
\sum_{C \in \Ch(\A)}t^{d(B,C)} = \prod_{i=1}^{\ell}(1+t+\dots + t^{d_{i}}), 
\end{align*}
where $ (d_{1}, \dots, d_{\ell}) $ is the exponents of $ \A $ and $ d(B,C) = |S_{B}(C)| $ denotes the number of hyperplanes in $ \A $ separating $ C $ from $ B $. 
\end{theorem}

We will give a canonical chamber for $ \I $ concretely. 
\begin{theorem}
Let $ N = (N_{2}, N_{3}, \dots, N_{\ell}) $ with $ N_{2} \supseteq N_{3} \supseteq \dots \supseteq N_{\ell} $.
Then the chamber 
\begin{align*}
B_{\ell} := \left(\bigcap_{2 \leq j \leq \ell} \{ x_{1} - x_{j} < \min N_{j} \}\right) \cap \left\{x_{2} < x_{3} < \cdots < x_{\ell}\right\} \cap \{ z > 0 \} \in \Ch(\I)
\end{align*}
is canonical. 
\end{theorem}
\begin{proof}
First, we show that $ B_{\ell} $ is a chamber of $ \I $. 
Indeed, $ B_{\ell} $ can be expressed as the intersection of half spaces with respect to all hyperplanes in $ \I $: 
\begin{align*}
B_{\ell} = \left(\bigcap_{\substack{2 \leq j \leq \ell \\ a \in N_{j}}} \{ x_{1} - x_{j} < a \}\right) \cap \left( \bigcap_{2 \leq j < k \leq \ell} \{ x_{j} - x_{k} < 0 \} \right) \cap \{ z > 0 \}. 
\end{align*}
Therefore $ B_{\ell} $ is a chamber or the empty set. 
Since the point
\begin{align*}
(x_{1}, x_{2}, \dots, x_{\ell}, z) = \left(1+\min N_{2}, 2, \dots, \ell, 1 \right)
\end{align*}
belongs to $ B_{\ell} $, we have that $ B_{\ell} $ is not empty. 
Thus $ B_{\ell} $ is a chamber of $ \I $. 

Next, we will prove that $ B_{\ell} $ is canonical by induction on $ \ell $. 
By definition, $ B_{2} $ is canonical. 
Assume that $ \ell \geq 3 $. 
Let $ \I_{\ell - 1} $ be the subarrangement of $ \I $ as defined in Section 2, namely 
\begin{multline*}
\I_{\ell - 1} 
= \left\{ \{x_{1}-x_{j} = az \} \mid 2 \leq j \leq \ell - 1 \text{ and } a \in N_{j} \right\} \\
\cup \left\{ \{ x_{j} - x_{k} = 0 \} \mid 2 \leq j < k \leq \ell - 1 \right\}
\cup \left\{ \{ z = 0\} \right\}. 
\end{multline*}
Since $ \I_{\ell - 1} $ is equivalent to $ \c(\Is{N_{2}, \dots, N_{\ell-1}}) $, the chamber $ B_{\ell-1} \in \I_{\ell - 1} $ is canonical by the induction hypothesis. 
An inclusion $ B_{\ell} \subseteq B_{\ell -1} $ yields $ \pi_{\ell}(B_{\ell}) = B_{\ell - 1} $. 
Therefore it suffices to show that $ \mathcal{F}(B_{\ell}) $ is linearly ordered. 
For a chamber $ C \in \mathcal{F}(B_{\ell}) $, every hyperplane $ H \in \I_{\ell-1} $ does not belong to $ S_{B_{\ell}}(C) $ since $ \pi_{\ell}(C) = \pi_{\ell}(B_{\ell}) $, i.e., $ C $ is on the same side as $ B_{\ell} $ with respect to $ H $. 

For every hyperplane $ \{x = k\} \in \I \setminus \I_{\ell-1} $, define a subset of $ \mathcal{F}(B_{\ell}) $ by
\begin{align*}
[x < k] := \left\{ C \in \mathcal{F}(B_{\ell}) \mid C \subseteq \{x < k \} \right\}.
\end{align*}
Let $ N_{\ell} = \{ a_{1}, a_{2}, \dots, a_{n} \} \subseteq \R $ with $ a_{1} < a_{2} < \dots < a_{n} $. 
Then we have a sequence related to the all hyperplanes in $ \I \setminus \I_{\ell-1} $: 
\begin{align*}
[x_{\ell-1} - x_{\ell} < 0] \subseteq [x_{\ell-2} - x_{\ell} < 0] \subseteq \dots \subseteq [x_{2} - x_{\ell} < 0] \\
\subseteq [x_{1} - x_{\ell} < a_{1}z] \subseteq [x_{1} - x_{\ell} < a_{2}z] \subseteq [x_{1} - x_{\ell} < a_{n}z]. 
\end{align*}
The inclusions except for $ [x_{2} - x_{\ell} < 0] \subseteq [x_{1} - x_{\ell} < a_{1}z] $ are obvious. 
To verify the exception,  let $ b \in N_{2} $ be the minimum element. 
By the assumption $ N_{2} \supseteq N_{\ell} $, we have $ b \leq a_{1} $. 
Since every chamber in $ \mathcal{F}(B_{\ell}) $ is in $ \{ x_{1} - x_{2} < bz \} \cap \{ z > 0 \} $, we have 
\begin{align*}
x_{2}-x_{\ell} < 0,\;  x_{1} - x_{2} < bz,  \; z > 0
&\Rightarrow x_{1} - x_{\ell} < bz, \; z > 0 \\
&\Rightarrow x_{1} - x_{\ell} < a_{1}z, 
\end{align*}
which implies $ [x_{2} - x_{\ell} < 0] \subseteq [x_{1} - x_{\ell} < a_{1}z] $. 

From the sequence above, we may say that $ [x_{\ell-1} - x_{\ell} < 0] = \{B_{\ell}\} $ and $ S_{B_{\ell}}(C) $ and $ S_{B_{\ell}}(C^{\prime}) $ are comparable for any chambers $ C, C^{\prime} \in \mathcal{F}(B_{\ell}) $, so are $ C $ and $ C^{\prime} $. 
Thus $ \mathcal{F}(B_{\ell}) $ is linearly ordered. 
Hence $ B_{\ell} $ is a canonical chamber. 
\end{proof}

\begin{cor}
For the Ish arrangement $ \Is{\ell} $, 
let $ B := \left\{ x_{1} < x_{\ell} < \dots < x_{2} \right\} $. 
Then
\begin{align*}
\sum_{C \in \Ch(\Is{\ell})} t^{d(B,C)} = (1+t+t^{2}+ \dots + t^{\ell})^{\ell - 1}. 
\end{align*}
\end{cor}

\section{A basis for $D(\I)$}

In this section, we will prove Theorem \ref{maintheoB}.
First, we verify that $\theta_{0},\theta_{1},\ldots ,\theta_{\ell}$
belong to $D(\I)$.

\begin{lemma}
\label{lemoftheoB}
Let $ N = (N_{2}, N_{3}, \dots, N_{j}) $ with $ N_{2} \subseteq N_{3} \subseteq \dots \subseteq N_{j} $. 
Then
\begin{align*}
\theta_{0} &= \sum_{i=1}^{\ell} \p{x_{i}}, \qquad
\theta_{1} = 
\left(
\sum_{i=1}^{\ell}x_{i}\frac{\partial}{\partial x_{i}}
\right)
+ z\frac{\partial}{\partial z},   \\
\theta_{k} &= \sum_{s=2}^{k}
\left(
\prod_{a \in N_{k}} (x_{1}-x_{s}-az)\prod_{t=k+1}^{\ell}(x_{s}-x_{t})
\right)
\frac{\partial}{\partial x_{s}}
\ \ (2 \leq k \leq \ell)
\end{align*}
belong to $D(\I)$.
\end{lemma}

\begin{proof}
Since $\theta_{0}(\alpha_{H}) =0$ for any $H \in \I$,
it belongs to $D(\I)$.
The Euler derivation $\theta_{1}$ belongs to 
$D(\A)$ for any central arrangement $\A$, thus $\theta_{1} \in D(\I)$.
We will show that $\theta_{k} \in D(\I)$ for
$2 \leq k \leq \ell$.
It is obvious that $\theta_{k}(z) = 0 \in zS$.

Let $2 \leq i<j \leq \ell$.

Case 1.
If $i<j \leq k$, then
\begin{align*}
\theta_{k}(x_{i}-x_{j})
&=\left(
\prod_{a \in N_{k}} (x_{1}-x_{i}-az)\prod_{t=k+1}^{\ell}(x_{i}-x_{t})
\right) \\
&\hspace{15ex}
-\left(
\prod_{a \in N_{k}} (x_{1}-x_{j}-az)\prod_{t=k+1}^{\ell}(x_{j}-x_{t})
\right) \\
&\equiv \left(
\prod_{a \in N_{k}} (x_{1}-x_{i}-az)\prod_{t=k+1}^{\ell}(x_{i}-x_{t})
\right) \\
&\hspace{15ex}
-\left(
\prod_{a \in N_{k}} (x_{1}-x_{i}-az)\prod_{t=k+1}^{\ell}(x_{i}-x_{t})
\right) \\
&\hspace{45ex}
(\mathrm{mod}\ x_{i}-x_{j}) \\
&=0,
\end{align*}
thus $\theta_{k}(x_{i}-x_{j}) \in (x_{i}-x_{j})S$.

\vspace{1em}
Case 2.
If $i \leq k < j$, then
\[
\theta_{k}(x_{i}-x_{j})=
\prod_{a \in N_{k}} (x_{1}-x_{i}-az)\prod_{t=k+1}^{\ell}(x_{i}-x_{t})
\in (x_{i}-x_{j}) S.
\]

Case 3.
If $k<i< j$, then
\[
\theta_{k}(x_{i}-x_{j})=0 \in (x_{i}-x_{j}) S.
\]
Hence $\theta_{k}(x_{i}-x_{j}) \in (x_{i}-x_{j}) S$ for
$2 \leq i<j \leq \ell$.

Let $2 \leq j \leq \ell$ and $b \in N_{j}$.

\vspace{1em}
Case 1.
If $j \leq k$, then $b \in N_{j} \subseteq N_{k}$, thus
\begin{align*}
\theta_{k}(x_{1}-x_{j}-bz) &=
\prod_{a \in N_{k}} (x_{1}-x_{j}-az)\prod_{t=k+1}^{\ell}(x_{j}-x_{t}) \\
&\in (x_{1}-x_{j}-bz)S.
\end{align*}

Case 2.
If $k < j$, then
\[
\theta_{k}(x_{1}-x_{j}-bz) = 0
\in (x_{1}-x_{j}-bz)S.
\]
Hence $\theta_{k}(x_{1}-x_{j}-bz) \in (x_{1}-x_{j}-bz) S$ for
$2 \leq j \leq \ell$ and $b \in N_{j}$.
Therefore we obtain that $\theta_{k} \in D(\I)$.
\end{proof}

\noindent
\textbf{Proof of Theorem \ref{maintheoB}.}

First, note that if $s=1,$ $k \geq 2$ then 
\[\theta_{k}(x_{s}) = \theta_{k}(x_{1}) =0,\]
and if $2 \leq k<s$ then
\[\theta_{k}(x_{s}) = 0.\]
Thus the determinant of the coefficient matrix of
$\theta_{0},\theta_{1},\ldots ,\theta_{\ell}$
can be calculated as follows:

\begin{align*}
\left|
\begin{array}{ccccc}
\theta_{0}(x_{1}) & \theta_{1}(x_{1}) & \cdots & \theta_{\ell}(x_{1}) \\
\vdots & \vdots & \cdots & \vdots \\
\theta_{0}(x_{\ell}) & \theta_{1}(x_{\ell}) & \cdots & \theta_{\ell}(x_{\ell}) \\
\theta_{0}(z) & \theta_{1}(z) & \cdots & \theta_{\ell}(z) \\
\end{array}
\right| 
&=
\left|
\begin{array}{ccccc}
1 & x_{1} & 0 & \cdots & 0 \\
1 & x_{2} & \theta_{2}(x_{2}) & \cdots & \theta_{\ell}(x_{2}) \\
\vdots & \vdots & 0 & \ddots & \vdots \\
1 & x_{\ell} & \vdots & \ddots & \theta_{\ell}(x_{\ell}) \\
0 & z & 0 & \cdots & 0
\end{array}
\right| \\
&\doteq z
\left|
\begin{array}{ccccc}
1 & 0 & 0 & \cdots & 0 \\
1 & \theta_{2}(x_{2}) & \theta_{3}(x_{2}) & \cdots & \theta_{\ell}(x_{2}) \\
1 & & \theta_{3}(x_{3}) & \cdots & \theta_{\ell}(x_{3}) \\
\vdots & & & \ddots & \vdots \\
1 & \multicolumn{2}{c}{\raisebox{1.5ex}[0pt]{\BigZero}} & & \theta_{\ell}(x_{\ell})
\end{array}
\right|
\end{align*}
\begin{align*}
&= z \prod_{k=2}^{\ell} \theta_{k}(x_{k}) \\
&=z \prod_{k=2}^{\ell}
\left(
\prod_{a \in N_{k}} (x_{1}-x_{k}-az)
\prod_{t=k+1}^{\ell}(x_{k}-x_{t}) 
\right) \\
&=z
\left(
\prod_{k=2}^{\ell} \prod_{a \in N_{k}} (x_{1}-x_{k}-az)
\right)
\left(
\prod_{k=2}^{\ell} \prod_{t=k+1}^{\ell}(x_{k}-x_{t})
\right) \\
&=z
\left(
\prod_{k=2}^{\ell} \prod_{a \in N_{k}} (x_{1}-x_{k}-az)
\right)
\left(
\prod_{2 \leq k < t \leq \ell}(x_{k}-x_{t})
\right) \\
&=Q(\I),
\end{align*}
where $\doteq$ denotes that they are
equal up to a nonzero constant multiple. 
Combining this calculation and Lemma \ref{lemoftheoB},
we can apply Saito's criterion \cite{stheory} and
see that $\theta_{0},\theta_{1},\ldots ,\theta_{\ell}$
form a basis for $D(\I)$. 

\section{Freeness of the deleted Ish arrangements}
Let $ K_{\ell} $ be the complete graph on $ \ell $ vertices. 
We can regard $ K_{\ell} $ as the set of directed edges $(i,j)\ (i<j)$, 
namely $ K_{\ell} = \left\{ (i,j) \mid 1 \leq i < j \leq \ell \right\} $. 
For a subgraph $ G \subseteq K_{\ell} $, Armstrong and Rhoades \cite{arish} defined the deleted arrangements $ \Sh{G} $ and $ \Is{G} $ and showed that they share many properties. 
In particular, it was proven that $ \Sh{G} $ and $ \Is{G} $ have the same characteristic polynomials by their explicit expressions. 
The deleted Shi and Ish arrangements are defined by
\begin{align*}
\Sh{G} := \Co{\ell} \cup \left\{ \{x_{i} - x_{j} = 1\} \mid (i,j) \in G \right\} \subseteq \Sh{\ell}, \\
\Is{G} := \Co{\ell} \cup \left\{ \{x_{1}-x_{j} = i \} \mid (i,j) \in G \right\} \subseteq \Is{\ell}. 
\end{align*}
Athanasiadis gave a necessary and sufficient condition for the freeness of $ \c(\Sh{G}) $. 

\begin{theorem}[{\cite[Theorem 4.1]{afree}} ]
\label{Athanasiadis}
Let $ G \subseteq K_{\ell} $ be a subgraph. 
The cone over the deleted Shi arrangement $ \c(\Sh{G}) $ is free 
if and only if 
there exists a permutation $ w $ of $ \{1, \dots, \ell\} $ such that 
$ w^{-1}G $ is contained in $ K_{\ell} $, i.e., 
$ (i,j) \in w^{-1}G $ implies $ i < j $, 
and has the following property: 
\begin{align*}
\text{ If } 1 \leq i < j < k \leq \ell \text{ and } (i,j) \in w^{-1}G \text{ then } (i,k) \in w^{-1}G. 
\end{align*} 
\end{theorem}

In this section, we will prove that the property of $ G $ in the Theorem \ref{Athanasiadis} is also a necessary and sufficient condition for the freeness of $ \c(\Is{G}) $ by making use of the terminology of the $ N $-Ish arrangements.
The problem of whether the cone of the deleted Ish arrangement $\c(\Is{G})$
is free or not is posed by Armstrong and Rhoades in 
\cite[p.\hspace{0.3ex}1517]{arish}
together with the problem for $\c(\Is{\ell})$.

For a subgraph $ G \subseteq K_{\ell} $, define a tuple of sets $ N_{G} =(N_{2}, \dots, N_{\ell}) $ by
\begin{align*}
N_{j} := \{ 0 \} \cup \left\{ i \mid (i,j) \in G \right\} \subseteq \{0, 1, \dots, j-1\}. 
\end{align*}
It is easy to show that $ \Is{N_{G}} = \Is{G} $. 

\begin{theorem}
\label{theoC}
Let $ G \subseteq K_{\ell} $ be a subgraph. 
Then the following are equivalent:
\begin{enumerate}
\item $ \c(\Is{G}) $ is free. 
\item $ N_{G} $ is a nest. 
\item $ G $ has the property in Theorem \ref{Athanasiadis}. 
\item For any $ j,k \in \{2, \dots, \ell\} $, either of the following two conditions holds:
\begin{enumerate}
\item[(i)] If $ (i,j) \in G $ then $ (i,k) \in G$ for any $ i \leq \min\{j,k\} $.
\item[(ii)] If $ (i,k) \in G $ then $ (i,j) \in G$ for any $ i \leq \min\{j,k\} $.
\end{enumerate}
\end{enumerate}
\end{theorem}
\begin{proof}
$ (1) \Leftrightarrow (2) $ 
It is obvious from Theorem \ref{maintheoA}. 

$ (2) \Rightarrow (3) $ 
Assume that $ N_{G} $ is a nest. 
Then there exists a permutation $ w $ of $ \{1,\dots,\ell\} $ with $ w(1)=1 $ such that
\begin{align*}
N_{w(2)} \subseteq N_{w(3)} \subseteq \dots \subseteq N_{w(\ell)}. 
\end{align*}

Now, we will prove that $ w^{-1}G \subseteq K_{\ell} $
i.e., $(i,j) \in w^{-1}G$ implies $i<j$. 
For any $ (i,j) \in w^{-1}G $, we have $ (w(i), w(j)) \in G $. 
Hence $ w(i) \in N_{w(j)} $. 
Then $ N_{w(j)} \not\subseteq N_{w(i)} $ since $ w(i) \not\in N_{w(i)} $. 
Since $ N_{G} $ is a nest, we have $ N_{w(i)} \subseteq N_{w(j)} $. 
Therefore $ i < j $, namely $ (i,j) \in K_{\ell} $. 
Thus we have showed that $ w^{-1}G \subseteq K_{\ell} $. 

Suppose that $ 1 \leq i < j < k \leq \ell $. 
Then we have a chain of implications: 
\begin{align*}
(i,j) \in w^{-1}G 
&\Rightarrow (w(i), w(j)) \in G 
\Rightarrow w(i) \in N_{w(j)} \\
&\Rightarrow w(i) \in N_{w(k)}
\Rightarrow (w(i), w(k)) \in G
\Rightarrow (i,k) \in w^{-1}G. 
\end{align*}
This proves that $ G $ satisfies the second condition. 

$ (3) \Rightarrow (4) $
Fix elements $ j,k \in \{2, \dots, \ell\} $ and assume that $ w^{-1}(j) < w^{-1}(k) $. 
For any $ (i,j) \in G $, we have that $ (w^{-1}(i), w^{-1}(j)) \in w^{-1}G $. 
Since $ w^{-1}G \subseteq K_{\ell} $, we have $ w^{-1}(i) < w^{-1}(j) $. 
Then the second property of (3) implies $ (w^{-1}(i), w^{-1}(k)) \in w^{-1}G $, i.e., $ (i,k) \in G $. 
Hence (i) holds. 
Similarly, if $ w^{-1}(j) > w^{-1}(k) $ then (ii) holds. 

$ (4) \Rightarrow (2) $
For any $ j,k \in \{2, \dots, \ell\} $, it is clear that (i) holds if and only if $ N_{j} \subseteq N_{k} $ and (ii) holds if and only if $ N_{k} \subseteq N_{j} $. 
Therefore every element in $ N_{G} $ is comparable. 
Thus $ N_{G} $ is a nest. 
\end{proof}

Combining Theorem \ref{Athanasiadis} and Theorem \ref{theoC}, we can prove that the following corollary: 

\begin{cor}
The deleted arrangements $ \Sh{G} $ and $ \Is{G} $ share the freeness. 
\end{cor}

\section*{Acknowledgments}
The authors are grateful to Brendon Rhoades for suggesting a base chamber
for the wall-crossing formula of the original Ish arrangement $\Is{\ell}$.
The proof of Theorem 3.3 is based on his idea.

\bibliographystyle{amsplain1}
\providecommand{\bysame}{\leavevmode\hbox to3em{\hrulefill}\thinspace}
\providecommand{\MR}{\relax\ifhmode\unskip\space\fi MR }
\providecommand{\MRhref}[2]{%
  \href{http://www.ams.org/mathscinet-getitem?mr=#1}{#2}
}
\providecommand{\href}[2]{#2}

\end{document}